\documentclass[12pt]{article}

\newlength{\stefan}
\usepackage{amsmath, amsthm, amsfonts, makeidx}
\usepackage{multind}
\DeclareMathSymbol{\subsetneq}{\mathord}{AMSb}{"26}

\newtheorem{lemma}{Lemma}[section]

\newtheorem{theorem}[lemma]{Theorem}

\newtheorem{corollary}[lemma]{Corollary}
\theoremstyle{definition}

\newtheorem{definition}[lemma]{Definition}
\newtheorem{example}[lemma]{Example}

\def\g{{\gamma}}

\newcommand{\lp}{\longrightarrow}

\newcommand{\mb}{\mathbb}
\newcommand{\m}{\mathsf{m}}

\newcommand{\X}{\mathcal{X}}

\newcommand{\C}{\mb{C}}

\newcommand{\N}{\mb{N}}

\newcommand{\desda}{\Longleftrightarrow}

\renewcommand{\m}{{\mathfrak{m}}}

\newcommand{\Arg}{\mathrm{Arg}}
\newcommand{\Jac}{\operatorname{Jac}}
\newcommand{\MA}{\operatorname{MA}}
\newcommand{\GA}{\operatorname{GA}}

\title{Maps that are roots of power series}
\author{Stefan Maubach\footnote{Funded by Veni-grant of council for the
physical sciences, Netherlands Organisation for scientific research (NWO)}\\ \ \\
\small
Radboud University Nijmegen\\\small Toernooiveld 1, The Netherlands\\ \small s.maubach@math.ru.nl\\
\ \\
Han Peters\\
\ \\
\small
University of Wisconsin\\
\small 480 Lincoln Drive, United States\\
\small peters@math.wisc.edu}

\begin{document}

\maketitle

\begin{abstract}
We introduce a class of polynomial maps that we call polynomial roots of powerseries, and show that
automorphisms with this property generate the automorphism group in any dimension.
In particular we determine generically which polynomial maps that preserve the origin are roots of powerseries.
We study the one-dimensional case in greater depth. 
\end{abstract}

\section{Introduction}

For linear maps we have an algebraic formula that tells us exactly when the linear map is invertible: the determinant. We also have a closed algebraic formula that
gives us a polynomial, the characteristic polynomial, which gives us even more information about the linear map.

If one has a polynomial map $F:\C^n\lp \C^n$ the hope is that there exist similar closed formulas.
One such hope is the Jacobian Conjecture (see \cite{Essenboek}):
\[ F \textup{~invertible~} \desda \det\Jac(F)\in \C^* \]
where the ``$\Longleftarrow$'' is the actual conjecture.
Apparently a lot of polynomial maps do not contradict this statement: in fact, all that have been tried up to this point. But we are completely
in the dark {\em why} this formula works in so many (and maybe all) cases. After all, similar statements, like the Jacobian Conjecture for reals (see \cite{Pin94})
and the Jacobian Conjecture for dynamical systems (the Markus-Yamabe conjecture, see\cite{CEGHM97}) are not true.
But, inspired by the success of this formula,
it is an interesting question if there exist other closed formulas that give information on a polynomial map (like whether it is invertible).

While this natural extension of the determinant is well-known, before the paper \cite{FM07}
no attempt had been made at finding a closed formula for a characteristic polynomial for polynomial
endomorphisms, i.e. a polynomial $\sum_{i=0}^D a_iT^i$ which has as input the degree $d$ of $F$, the coefficients of $F$, and as output $D$ and the complex numbers $a_i\in
\C$.
At first it seems strange that such characteristic polynomials were not studied before 2007, until one realizes that
most polynomial endomorphisms do not have a nonzero characteristic polynomial having coefficients
in $\C$.
Those that do have a nonzero characteristic polynomial are called locally finite polynomial endomorphism (short: LFPE or LF map).
We will elaborate on them quickly:

A polynomial map $F:\C^n\lp \C^n$ is called a {\em locally finite polynomial endomorphism} (short LFPE) if it is a ``root'' of a nonzero polynomial
$p(T):=\sum_{i=0}^d a_i T^i$, which in turn means
that $\sum_{i=0}^d a_i F^i=0$ (where $F^i=F\circ F\circ \cdots \circ F$).
Though several subclasses of these maps were studied before (for example, maps that satisfy $F^s=I$ for some $s\in \N$ \cite{Kr96},
or maps that satisfy $F^2-2F+I=0$ in \cite{Bon}), the article \cite{FM07} is the first comprehensive study of LFPE, giving
an explicit description for a characteristic polynomial of such maps.
On a side note, it turns out that LFPE's share more properties with linear maps than generic polynomial endomorphisms (or automorphisms), in some sense.
Additionally, it may be that LFPE's form a natural generating set for the polynomial automorphism group of $\C^n$. In $\C^2$ the group of polynomial automorphisms is
generated by elementary automorphisms \cite{Ju42}, \cite{vdk53} but it was shown recently \cite{SU03} that the elementary maps do not generate the automophism group in dimensions $3$ and higher. Currently no non-trivial set of generators is known. We note that elementary automorphisms are locally finite, as well as the the Nagata automorphism (an automorphism that is not generated by elementary automorphisms, see \cite{SU03}).

As the set of locally finite polynomial endomorphisms is still rather small, it would be worthwhile to extend this definition. The goal is to find a class that is large
enough to generate the whole group of polynomial automophisms, yet small enough such that maps in this class can be understood more easily. Finding such a set of generators
could lead to a much better understanding of the polynomial automorphisms groups in dimensions $3$ and higher.

One noteworthy attempt at extending is the introduction of so-called quasi-LFPEs (in \cite{Fur07}),
where maps $F$ are studied which are ``zero'' of a polynomial with coefficients
that are in
\[\C(X)^F:=\C(X_1,\ldots,X_n)^F:=\{p\in \C(X_1,\ldots,X_n)~|~p(F)=p\}.\]
So for a polynomial endomorphism $F$ there must be
$a_i\in\C(X)^F$ such that $\sum_{i=0}^d a_iF^i=0$. Interesting is that the set of LFPE's
is exactly the set of polynomial endomorphisms for which the sequence
$\{deg(F^n)\}_{n\in\N}$ is bounded, whereas the set of quasi-LFPEs is contained in the set of endomorphisms
for which the sequence $\{deg(F^n)\}_{n\in\N}$ is bounded by a linear sequence in $n$. Even though this class is strictly larger, many polynomial endomorphisms do not show such a linear
growth of degree.

In this article we will study an alternative generalization of locally finite maps: polynomial endomorphisms that are roots of a non-negative {\emph{power series}} with coefficients in $\C$. We will make this more precise in Section \ref{notations}. In Section \ref{MVC} we will prove that this is a very large class of endomorphisms, and in particular the polynomial automorphisms satisfying this condition form a generator set.

In an attempt to get a better understanding of roots of powerseries, we will study in Section \ref{OVC} which polynomials (in the complex plane) are roots of powerseries.

The authors would like to thank prof. Jelonek for coming up with the question of studying polynomial maps which are roots of power series. 

\section{Notations}\label{notations}

We write $\MA_n(\C)$ for the monoid of polynomial endomorphisms of $\C^n$. We write $\GA_n(\C)$ for the set of polynomial automorphisms.
(Both notations are inspired by the idea that they are extensions of linear maps, one of the monoid of linear maps, $\operatorname{ML}_n(\C)$, and the other of the set of invertible
linear maps, $\operatorname{GL}_n(\C)$.)\\
$\C^{[n]}$ will denote the polynomial ring in $n$ variables.
If $F\in \MA_n(\C)$, then we write  $F^i:=F\circ F\circ \cdots \circ F$. Now every power $F^i$ has for each $1\leq k\leq n$ components $F^i_k\in \C^{[n]}$.
If $v\in \N^{n}$, we write $X^v:=X_1^{v_1}X_2^{v_2}\cdots X_n^{v_n}$ and $|v| = v_1 + v_2 + \ldots + v_n$. The coefficient of the term $X^v$ in $F^i_k$ we shall denote by $F^i_{(k,v)}$.

\begin{definition}\label{def1} If $F^0,F^1,F^2,F^3,\ldots,$ is any sequence of elements of $MA_n(\C)$, then we will say that
\[ \sum_{i=0}^{\infty} a_iF^i =0 \textup{~if~} \sum_{i=0}^{\infty} a_iF^i_{(k,v)}=0 \textup{~for~all~} 1\leq k\leq n, v\in \N^n. \]
If $F \MA_n(\C)$ and there exist complex numbers $a_o, a_1, \ldots$ such that $\sum_{i=0}^{\infty} a_iF^i =0$ (where $F^0$ is the identity map) then we will say that $F$ is a {\emph{root of a powerseries}}.
\end{definition}

Note that we do not require that $\sum_{i=0}^{\infty} a_i z^i$ has radius of convergence greater than $0$.

While the above definition of $\sum_{i=0}^{\infty} a_iF^i =0$ is natural when one considers a polynomial endomorphism as an element of $\C^\N$, where the coefficients give
the coordinates. If one considers polynomial endomorphisms purely as maps from $\C^n$ to $\C^n$ then it is to define that $\sum_{i=0}^{\infty} a_iF^i =0$ if the maps
$\sum_{i=0}^{N} a_iF^i$ converge (uniformly in a neighborhood of the origin) to the zero-map as $N \rightarrow \infty$. 
The second author will study this strictly stronger
definition in 
\cite{Pe03}.

For $F \in MA_n(\C)$ and $d \in \C$ we write $[F]_d$ for the $d$-jet of $F$, i.e. the polynomial endomorphism obtained by ignoring all terms of $F$ of degree $d+1$ and higher.

\begin{definition}\label{def2}
We will say that $F$ is a zero of a polynomial $P(T):=\sum_{i=0}^m p_iT^i$ {\emph{up to degree $d$}} if $\sum_{i=0}^m p_i [{F}^i]_d =0$.
\end{definition}
\section{The multi-variable case} \label{MVC}

Let $F \in MA_n(\C)$, and assume that $F(0) = 0$. In this section we will see that whether $F$ is a root of a powerseries depends almost entirely on the linear part of $F$. As a direct corollary we will see that the polynomial automorphism group is generated by roots of powerseries.

In order to know whether there exist $a_1, a_2, \ldots$ such that

\[ \sum_{i=0}^{\infty} a_iF^i_{(k,v)}=0 \textup{~for~all~} 1\leq k\leq n, v\in \N^n, \]

we need a good description of $F^i_{(k,v)}$ in terms of the coefficients of $F$ and the values $i,k,v$.

Fortunately, the description we need is already given in the proof of Theorem 1.2 in \cite{FM07}, although it is not stated explicitly in the theorem itself.
In theorem 1.2 in \cite{FM07}, the map $F$ is assumed to be LFPE, and therefore $deg(F^i)$ is bounded. Here, we will
not assume that $F$ is an LFPE but cut off $F^i$ at a certain degree, and the result is almost the same with exactly the same proof.

If $F$ is a zero of a polynomial $P$ up to degree $d$ (see definition \ref{def2}), then we can compose $\sum_{i=0}^m p_i F^i$ from the right with $F^j$ and see that even $\sum_{i=0}^m p_i \tilde{F}^{i+j}=0$ for each $j\in \N$.
This exactly means that the sequence $\{ \tilde{F}^{i}\}_{i\in \N}$ is a linear recurrent sequence with respect to the polynomial $P(T)$, and that also
\[ \{   F^i_{(k,v)} \}_{i\in\N} \]
is a linear recurrent sequence belonging to $P(T)$ for each $1\leq k\leq n$ and $v\in \N^n$ satisfying $|v|\leq d$.
Thus, we can obtain the following:

\begin{lemma}\label{LRS}
If $F$ is a zero of $P(T)$ up to a degree $d$, and let $\mu_1,\ldots,\mu_s$ be the roots of $P(T)$ and $e_1,\ldots,e_s$ be the multiplicities of these roots.
Then for each $1\leq k\leq n$ and each $v\in \N^n$ satisfying $|v|\leq d$, the sequence
\[ \{F^i_{(k,v)}\}_{i\in \N} \] is a $\C$-linear combination of the sequences
\[ \{i^u\mu_t^i\}_{i\in \N} \]
 where $t$ runs from $1$ to $s$, and $u$ runs from $0$ to $e_s-1$.
\end{lemma}

\begin{proof} As observed above, $\{  F^i_{(k,v)} \}_{i\in \N}$ is a linear recurrent sequence to the polynomial $P(T)$. By standard theory of linear
recurrent sequences, see for example \cite{CMP},
this means that the sequence $\{ F^i_{(k,v)}\}$ is a linear combination of the sequences $\{ i^b\mu_a^i\}_{i\in \N}$ where
$1\leq a\leq s$ and $0\leq b\leq e_a-1$.
\end{proof}

Let $\lambda_j$ where $1\leq j\leq n$ denote the eigenvalues of the linear part of $F$ (which do not have to be all different). For $v\in \N^n$ write
$\lambda^v:=\lambda_1^{v_1}\cdots \lambda_n^{v_n}$, and $|v|=v_1+v_2+\ldots+v_n$.

\begin{theorem}\label{Th1}
Let $F\in \MA_n(\C)$ be such that $F(0)=0$. Then
\[ \X_d:=\X_{F,d}(T):=\prod_{v\in \N^n, |v|\leq d} (T-\lambda^v) \]
is a vanishing polynomial of $F$ up to degree $d$, i.e. if $\X(T)=\sum a_iT^i$ then
$\sum a_i\tilde{F}^i=0$.
\end{theorem}

The proof is exactly the proof of theorem 1.2 in \cite{FM07}, and we refer to that paper for it. The basic ingredient is lemma \ref{LRS}.

\begin{example}
Let $\lambda, \gamma \in \C$ and consider the polynomial map $F:(z, w) \mapsto ( \lambda z + w^2, \gamma w)$. If $\lambda \neq \gamma^2$ then we have that
\[ [F^n]_2 (z,w) = (\lambda^n z + \frac{\lambda^n + \gamma^{2n}}{\lambda - \gamma^2} w^2, \gamma^n w). \]

In the exceptional case that $\lambda = \gamma^2$ we instead get

\[ [F^n]_2 (z,w) = (\lambda^n z + n\lambda^n w^2, \gamma^n w). \]
\end{example}

\

Now let $F\in \MA_n(\C)$ whose linear part has eigenvalues $\lambda_1,\ldots,\lambda_n$.
Theorem \ref{Th1} gives a vanishing polynomial $\X_{d}$  up to a certain degree, but it generally is not the {\em minimal} vanishing polynomial up to this degree.
Since the set of vanishing polynomials up to a certain degree forms an ideal of $\C[T]$ (see \cite{FM07} section I.1, or \cite{Mau03} theorem 4.3.3),
there exists a unique monic minimal vanishing polynomial which
we will denote by $\m_{F,d}(T)=\m_d(T)$. Note that $\m_d ~|~ \m_{d+1}$, as well as $\X_d | \X_{d+1}$.
These two polynomials help us in finding (all) power series $P$ for which $P(F)=0$. First, let us define a certain type of power series:

\begin{definition} We say that a power series $P$ is ``good for $F$'' if $P$ converges for each root $\mu$ of $\m_d$ for each $d$, and the multiplicity of the roots $\mu$ is at least that of $\m_d$ for each $d$. 
\end{definition}

\begin{lemma} \label{minimal}
(1) $P$ is good for $F$ $\Longrightarrow$ $P(F)=0$. \\
(2) If $\X_{F,d}=\m_{F,d}$ for all $d$, then 
$P$ is good for $F$ $\desda$ $P(F)=0$.
\end{lemma}

\begin{proof}
Write $P(T)=\sum p_i T^i$, fix $d$, and let $\mu$ be a zero of multiplicity $e$ in $\m_d$.
$P$ being good for $F$ implies that $\mu$  is a zero of $P$ for each $0\leq j\leq e-1$.
This is equivalent to $P^{(j)}(\mu)=0$ for all $0\leq j\leq e-1$, so
\[
\sum_{i=0}^{\infty} p_i  {i\choose j} \mu^{i-j} =0
\]
for all $0\leq j\leq e-1$. Multiplying by $\mu^j$ and doing some linear algebra, we get 
\[
\sum_{i=0}^{\infty} p_i  i^j \mu^{i} =0
\]
for all $0\leq j\leq e-1$.
 Now let us define $V_d$ as the $\C$-vector space generated by all the sequences $\{i^j\mu^i\}_{i\in \N}$, where $\mu$ runs over all roots of $\m_d$, and if $e$ is the multiplicity of this root, then $j$ runs from $0$ to  $e-1$.
Apparently, $P$ being good for $F$ is equivalent to: (*) any sequence $\{s_i\}_{i\in \N}$ in $V_d$ satisfies $\sum p_i s_i=0$. 

Define $W_d$ as the linear span of the sequences $\{F^i_{(k,v)}\}_{i\in \N}$ where $1\leq k \leq n, |v|\leq d$. 
Apparently, $P(F)=0$ is equivalent to: (**) any sequence $\{s_i\}_{i\in \N}$ in $W_d$ satisfies $\sum p_i s_i=0$.

In any case $W_d\subseteq V_d$, because of lemma \ref{LRS}. Because of (*) and (**), this proves (1). 
In case $\X_{F,d}=\m_{F,d}$ for all $d$, then $W_d=V_d$. This proves (2).
\end{proof}

It is unclear if part (1) of the above proof is really one-way, or can be improved to an if-and-only-if statement. 
For our needs, the above suffices.

So, heuristically speaking, we are looking for a  power series which is a limit of $\m_d$ as $d$ approaches infinity.
Such a power series need not exist:

\begin{example} \label{ex1} Let $F=X^2\in \MA_1(\C)$ (or any other nonzero $F\in \MA_n(\C)$ having linear part equal to zero).
Then there is no nontrivial power series $\sum_{i=0}^{\infty} a_iT^i$ such that $\sum_{i=0}^{\infty} a_iF^i=0$.
\end{example}

Another problem is that theorem \ref{Th1} gives a formula for $\X_d$, but not for $\m_d$. It is very well possible that there exists a power series having all
roots of $\m_d$, while there exists no power series having all the roots of $X_d$.
Nevertheless, \ref{Th1} is a helpful theorem:
it allows us to decide when there does {\em not} exist such a power series, and it covers the generic case.
The theorem asks for a power series that has $\lambda^v$ as a root for each $v\in \N^n$ (counting multiplicities).

We can rule out a few cases immediately.

\begin{lemma} \label{notroot}
Given $\lambda_1,\ldots, \lambda_n\in \C$.
Let $P$ be a power series that has for each $v\in \N^n$ a root $\lambda^v$. \\
(1) If $|\lambda_i|< 1$ for some $1\leq i\leq n$, then $P=0$.\\
(2) If there exist $i, j$ such that $|\lambda_i|=1$ and $|\lambda_j|> 1$, then $P=0$.\\
\end{lemma}

\begin{proof}
(1): In case $0<|\lambda_i|<1$, we have that $0$ is an accumulation point of the roots, as $\lambda_i^m$ is a root for each $m\in \N$.
The power series is defined in a positive radius around 0. So this means that
$P$ is zero. In the case that $\lambda_i=0$ then we get a root at 0 of order $\infty$. In case the power series is defined in a radius around 0, which is the case if there
exists some other value $\lambda_j\not = 0$ (as the power series has to be defined at $\lambda_j$) then $P$ must be zero. In case all the eigenvalues are
zero, then we are in the case of \ref{ex1}.\\
(2) Since $|\lambda_i| = 1$ the power series $P$ must have infinitely many roots on the unit circle (counting multiplicity). This means that there is either an accumulation point of these roots, or a root of infinite order. But since there is a $|\lambda_j|>0$, the radius of convergence of $P$ is strictly larger than 1, hence $P=0$.
\end{proof}

\begin{theorem}
Let $\mathcal{V}$ be the set of all $F \in \MA_n(\C)$ of fixed degree $d \ge 2$ with $F(0)=0$ and eigenvalues $\lambda_1, \ldots , \lambda_n$ 
of the linear part of $F$ that satisfy {\emph{either}} $|\lambda_i|< 1$ for some $1\leq i\leq n$ {\emph{or}} there exist $i, j$ such that 
$|\lambda_i|=1$ and $|\lambda_j|> 1$. Then a generic $F \in \mathcal{V}$ is not the root of a powerseries.
\end{theorem}
With generic we mean that the result hold for a countable intersection of open and dense subsets of $\mathcal{V}$. 
We note that we cannot expect the result to hold for any $F \in \mathcal{V}$, as any locally finite $F \in \mathcal{V}$ is a root of a powerseries.

\begin{proof}

First, define $U:=\{F\in \mathcal{V} ~|~$ the linear part of $F$ is diagonal $\}$. 
We will restrict ourselves to this set first. 

A jet $[F]_d$ of an $F\in U$ can be represented by an element in $\C^m$, where $m$ equals the amount of coefficients occuring in $F$ up and including degree $d$ (but ignoring the nondiagonal coefficients of the linear part). 
By theorem \ref{Th1} we know that, writing $\X_d=\sum s_i T^i$, that $\sum s_i [F^i]_d=0$. 
Both $s_i$ as well as $[F^i]_d$ are polynomial formulas in the coefficients of $[F]_d$, magically satisfying $\sum s_i [F^i]_d=0$. 
Now let $\X_{v}:=\X_d(T-\lambda^v)^{-1} =: \sum t_i T^i$. Define
\[ S_v:=\{[F]_d\in \C^m ~|~ \sum t_i [F^i]_d =0 \}. \]
$S_v$ is not all of $\C^m$, by lemma \ref{All}.
Because of this, and since it is defined by polynomial equations, $S_v$ is a Zariski closed set in $\C^m$ of codimension $\geq 1$. 
The set of all $[F]_d$ which have $\X_d$ as minimal polynomial equals
\[ S_d:=\C^m \backslash (\bigcup_{|v|\leq d} S_v) \]
and it is a nonempty Zariski open set in $\C^m$. Let $U_d:=\{~|~ [F]_d\in S_d\}$.
Define $V:=\cap_{d\in \N} U_d$. So, the  set of $F\in U$ which have $\X_d$ as minimum polynomial for each $d$ equals $V$ and  is generic in $U$. 
Conjugating by a {\em linear} map does not change any of the minimum polynomials $\m_d$. Let $\tilde{U}$ be all 
linear conjugates of elements in $U$, and $\tilde{V}$ be the linear conjugates of elements in $V$. Note that 
$\tilde{U}$ is the set of all $F\in \mathcal{V}$ which have linearizable linear part. Now we can state that 
the  set of $F\in \tilde{U}$ which have $\X_d$ as minimum polynomial for each $d$ equals $\tilde{V}$ and is generic in $\tilde{U}$.
$\tilde{U}$ is generic in $\mathcal{V}$, so $\tilde{V}$ is generic in $\mathcal{V}$.  
Using lemma \ref{minimal} part (2), we see that elements in $\tilde{V}$  do not have nonzero power series  
of which they are zeroes. 
\end{proof}

\begin{lemma}\label{All} $S_v\not = S$
\end{lemma}

\begin{proof}
For this it is enough to pick an appropriate $F\in MA_n(\C)$ such that $[F]_d\in \C^m$ is not a zero of 
$\mathcal{X}_v:=(T-\lambda^v)^{-1}\X$.
First the case that $d\geq |v|\geq 2$. Take $F=(\lambda_1 X_1 + X^{v}, \lambda_2 X_2, \ldots, \lambda_n X_n)$ where 
the $\lambda_i$ are such that $\lambda^w=1\lp w=0$. A simple proof shows that 
\[ [F^i]_v=(\lambda_1^i X_1 + \frac{\lambda_1^i -(\lambda^v)^i}{\lambda_1 - \lambda^v} X^v , \lambda_2^i X_2, \ldots, 
\lambda_n^i X_n) \]
which shows that $F$ satisfies  $\m_d(\lambda^v)=0$. $X_v$ does not have $\lambda^v$ as a root, thus $F$ is no zero of $X_v$.
Now the case $|v|=1$. One can assume that $v=(1,0,\ldots,0)$. Take $F$ a diagonal linear map, again its eigenvalues satisfying $\lambda^w=1\lp w=0$.
Any polynomial not having $\lambda_1$ as zero, will not have $F$ as a zero. $(T-\lambda_1)^{-1}\X$ is such a polynomial. 
\end{proof}

We will now prove a similar results that give that, under different conditions for the eigenvalues, $F$ is a root of a powerseries.

\begin{lemma} \label{expanding}
Let $\lambda_1,\ldots, \lambda_n\in \C$ and assume that {\emph{either}} $|\lambda_i| >1$ for every $i$ {\emph{or}} $|\lambda_i|=1$ for every $i$. Then there exists a powerseries $P$ that has for each $v\in \N^n$ a root $\lambda^v$.
\end{lemma}
\begin{proof}
The case where $|\lambda_i| >1$ follows immediately from the Weierstrass product theorem.

So let us assume that $|\lambda_i|=1$ for every $i$. Let $v_1, v_2, \ldots$ be a linear ordering of the $v \in \N^n$. For a sequence $n_1, n_2, \ldots \in \N$ define the polynomials
\begin{eqnarray*}
P_N(T) = \prod_{i=1}^N \left(1- ( 1 - \frac{T}{\lambda^{v_j}} )^\frac{1}{2} \right)^{n_j}.
\end{eqnarray*}

If the sequence $\{n_j\}$ increases fast enough then the polynomials converge uniformly on compact subsets of the open unit disc. Therefore the limit function $P(T)$ is holomorphic and given by a powerseries. The powerseries necessarily has radius of convergence $1$ and can be made to convergence on the unit circle by choosing the $n_j$'s large enough. It has roots at each $\lambda_v$ with the required multiplicities.
\end{proof}

\begin{theorem} \label{main}
Let $F \in \MA_n(\C)$ with $F(0)=0$, and let $\lambda_1, \ldots , \lambda_n$ be the eigenvalues of the linear part of $F$. 
If {\emph{either}} $|\lambda_i|> 1$ for all $1\leq i\leq n$ {\emph{or}} $|\lambda_i|=1$ for all $1\leq i\leq n$, then $F$ is the root of a powerseries.
\end{theorem}
\begin{proof}
This follows immediately from Lemmas \ref{minimal} and \ref{expanding}.
\end{proof}

\begin{corollary}
The polynomial automorphisms that are roots of powerseries generate the polynomial automorphism groups.
\end{corollary}
\begin{proof}
If $F$ is a polynomial automorphism then we can find an invertible affine map $A$ such that $G=AF$ maps $0$ to $0$ and such that the eigenvalues of the linear part of $G$ have absolute value strictly greater than 1. Hence $G$ is a root of a power seriers and so is $A^{-1}$ (as it is affine and thus locally finite) and $F = A^{-1}G$.
\end{proof}

So not only do roots of power series generate the automorphism groups, in fact any polynomial automorphism is a composition of an affine map and a root of a powerseries. Clearly this generalization of locally finite automorphisms is too general to obtain a better understanding of the polynomial automorphism groups.

\section{The one variable case}\label{OVC}

In the previous section we only considered endomorphisms $F$ with $F(0) = 0$. Here we will consider polynomials $f(z)$ and we do not require that $f(0) =0$. The authors do not know if the results obtained in this section hold in higher dimensions. It will be clear that the methods we use only work in the one-dimensional case.

Let us first observe that whether $f(z)$ is a root of a powerseries is invariant under conjugation by a linear function $z \mapsto az$ with $a \neq 0$. However, we will show
(see Corollary \ref{translation} below) that being a root of a powerseries is not invariant under conjugation by affine functions. In fact, we will show that given $f(z)$ we can always find a translation $\tau(z) = z+c$ such that $\tau \circ f \circ \tau^{-1}$ is a root of a powerseries.

Then we will show that if $f = b_0 + b_1 z + \cdots b_d z^d$ with $b_1 \neq  0$ and we are allowed to change the constant term $b_0$, then we can make sure that $f$ is a root of a powerseries (see Theorem \ref{allroots} below). We can think of this in the following way: every polynomial (with non-zero linear coefficient) is a root of a powerseries about some $c \in \C$, i.e.

\[ \sum_{i=0}^\infty a_i (f(z) - c)^n = 0. \]

Let us recall that if $f$ has a fixed point $w$ (i.e. $f(w) = w$) then $w$ is called an attracting fixed point if $|f^\prime(w)|<0$, a neutral fixed point if $|f^\prime(w)| = 1$ and a repelling fixed point if $|f^\prime(w)|>1$. It is a well known fact that every polynomial of degree at least $2$ has a fixed point that is not attracting, we give a short proof:

\begin{lemma} \label{repelling}
Let $f(z)$ be a polynomial of degree $d \ge 2$. Then $f$ has at least one neutral or repelling fixed point.
\end{lemma}
\begin{proof}
The fixed points of $f$ are the solutions of the equation $f(z) - z = 0$, so there are $d$ fixed points $r_1, \ldots , r_d$ counting multiplicity. If a fixed point has multiplicity $2$ or higher then the derivative at the fixed point must be equal to $1$, so we may assume that $r_1, \ldots , r_d$ are distinct.

Now consider the contour integral

\[ \frac{1}{2\pi i}\int_{C} \frac{1}{f(z) - z} dz,\]

where $C$ is a large circle oriented counterclockwise such that $r_1, \ldots , r_d$ all lie inside the circle. Since $f$ is of degree $d \ge 2$ we have that this integral is equal to $0$, and it follows from the residue theorem that

\[ \sum_{j=1}^{d} \frac{1}{f^\prime(r_j) - 1} = 0.\]

Hence it follows that for some $1 \le j \le d$ we have $\mathrm{Re}(f^\prime(r_j) - 1) \ge 0$. But then $|f^\prime(r_j)| \ge 1$.
\end{proof}

\begin{corollary}\label{translation}
If $f$ is a polynomial then there exist a translation $\tau(z) = z+c$ such that $\tau \circ f \circ \tau^{-1}$ is a root of a powerseries.
\end{corollary}
\begin{proof}
Any polynomial of degree $1$ is already a root of a powerseries, so we may assume that the degree is greater or equal to $2$. But then it follows from Lemma \ref{repelling} that there exist a fixed point $c$ with $|f^\prime(c)|\ge 1$. After conjugation by $\tau$ we have that the polynomial is of the form $p(z) =  \lambda z + h.o.t.$ with $|\lambda| \ge 1$, and it follows from Theorem \ref{main} that $p$ is a root of a powerseries.
\end{proof}

We saw in the previous section that if $f(z)$ is a polynomial with $f(0) = 0$, then we have a very good description of all the coefficients of the polynomials $f^i(z)$. However, it is much harder to understand what the coefficients of the functions $f^i$ are when $f(0) \neq 0$. To prove Theorem \ref{allroots} below we need estimates on the size of all these coefficients. Fortunately we will see that if we are careful about picking the constant term then we do get good enough estimates that allow us to prove that $f$ is a root of a powerseries.

In the proof of Theorem \ref{allroots} we will use the following technical lemma.

\begin{lemma}\label{orthogonal}
Let $\{b_j^i\}_{j, i \in \N}$ and suppose that for every $j \in \N$ and $c > 0$ we have that for all sufficiently large $i \in \N$ the inequality $|b_j^i| \ge c |b_k^i|$ holds for $1 \le k \le j-1$. Then there exists a sequence $a_1, a_2, \ldots$ with
$$
\sum_{i=1}^{+\infty} a_i b_j^i = 0,
$$
for every $j \in \N$.
\end{lemma}

\begin{proof}

If $b_1^i = 0$ for infinitely many $i \in \N$ then we can restrict ourselves to those $i$ and the condition $\sum_{i=1}^{+\infty} a_i b_1^i = 0$ is automatically satisfied. We can then ignore the $b_1^i$ and continue with the $b_2^i$, renaming them $b_1^i$. Hence we may assume that $b_1^i \neq 0$ for all but finitely many $i \in \N$, and by restricting to a subsequence we may assume that $b_1^i \neq 0$ for any $i \in \N$. We will also assume that the sequence $|b_1^i|$ is increasing and converges to infinity, this is just a matter of rescaling the $b_i^j$'s and the $a_i$'s.

We will construct the sequence $a_1, a_2, \ldots$ recursively. First we choose $a_1$ and $a_2$ such that $a_1 b_1^1 + a_2 b_1^2 = 0$.
Now suppose that we have chosen $a_1, a_2, \ldots , a_{N_k}$ such that

$$
\sum_{i=1}^{N_k} a_i b_j^i = 0,
$$

for every $j \le k$. We claim that we can choose $a_{N_k + 1},
a_{N_k+ 2} ,\ldots , a_{N_{k+1}}$ such that the following are
satisfied:

\noindent (i) $a_i = 0$ for all but $k+1$ choices of $i \in [N_k
+ 1, N_{k+1} ]$.

\noindent (ii) $|a_i b_j^i| \le \frac{1}{2^k}$ for $i \in
[N_k + 1, N_{k+1} ]$ and $j = 0, \ldots k$. And,

\noindent (iii) $\sum_{i=1}^{N_{k+1}} a_i b_j^i = 0$ for every $j \le K=1$.

We recursively continue the construction of the sequence $a_1,
a_2, \ldots$. It follows from (i) and (ii) that the sum
$\sum_{i=1}^{+\infty} a_i b_j^i$ converges for every $j$, and
(iii) completes the proof.
To prove the claim, let $C = \sum_{i=1}^{N_k} a_i b_{k+1}^i$. The claim 
follows automatically when $C=0$ so we may assume $C\neq 0$. We pick $N_k < n_1 < n_2 < \cdots < n_k$ such that the $k\times k$
matrix $(b_j^{n_l})$ has full rank (an easy induction argument
shows that this is possible). Let $u=(u_1, \ldots, u_k)$ be the
unique vector with $b_{k+1}^{n_l} = u\cdot(b_1^{n_l}, \ldots ,
b_k^{n_l})$ for every $l \le k$.
Since $|b_{k+1}^i| > c |b_j^i|$ holds for all $1 \le j \le k$, we can make sure that $n_{k+1}$ is large enough such that

\begin{align} \label{estimate1}
|u \cdot (b_1^{n_{k+1}}, \ldots, b_k^{n_{k+1}})| \le \frac{1}{2} |b_{k+1}^{n_{k+1}}|.
\end{align}

Also, since the sequences $|b_j^i|$ are all eventually increasing and unbounded we can pick $n_{k+1}$ large enough such that

\begin{align}\label{estimate2}
|b_j^{n_{k+1}}| \ge k |b_j^{n_l}|.
\end{align}

Let

$$
M = \min_{v} \max_{j} |\sum_{l=1}^{k} b_j^{n_l} v_l|,
$$

where the minimum is taken over all unit vectors in $\C^k$, and
the maximum is taken over $j \le k$. Also define $K = 1 + \max
|b_j^{n_l}|$ where the maximum is taken over $j, l \le k$.
Now pick $n_{k+1}$ such that for $j = 1, \ldots , k$  we have
$b_j^{n_{k+1}} \neq 0$  and

\begin{align}\label{estimate3}
|b_{k+1}^{n_{k+1}}| > \tilde{C} b_j^{n_{k+1}},
\end{align}

where $\tilde{C} = \frac{2^{k+2} |C| K}{M}$. Let $v = (v^\prime , v_{k+1}) \in \C^{k+1}$ be a unit vector with
$\sum_{l=1}^{k+1} v_l b_j^{n_l} = 0$ for every $j = 1, \ldots k$.
By inequality \eqref{estimate2} we have that $v_{k+1} \le \frac{1}{2}$, so we have that for some $j$ the following hold:

\begin{align}\label{estimate4}
|v_{k+1} b_j^{n_{k+1}}| = |\sum_{l=1}^k v_l b_j^{n_l}| \ge \frac{M}{2}.
\end{align}

Therefore we have that

\begin{align*}
|\sum_{l=1}^{k+1} v_{l} b_{k+1}^{n_l}| \ge \frac{1}{2}|v_{k+1} b_{k+1}^{n_{k+1}}|
\ge \frac{1}{2} \max|v_{k+1} b_j^{n_{k+1}}| \tilde{C} \ge
\frac{M}{4}\tilde{C}.
\end{align*}

Here the first inequality uses \eqref{estimate1}, the second inequality uses \eqref{estimate3}, and the last inequality uses \eqref{estimate4}. Therefore
\[ |\sum_{l=1}^{k+1} v_l  b_{k+1}^{n_l}| \ge 2^k |C|K. \]
If we normalize the vector $v$ such that
$\sum v_l b_{k+1}^{n_l} = -C$ then it follows from our definition of $K$ that
$|v_l b_j^{n_l}| < \frac{1}{2^k}$ for every $j \le k$ and $l\le k+1$.
So if we take $N_{k+1} = n_{k+1}$ and choose the $a_{n_l} =
v_l$ and the rest of the $a_i$ equal to $0$ then the conditions of the
claim are all satisfied.

\end{proof}

Using this lemma we can prove the following theorem.

\begin{theorem}\label{allroots}
Let $f(z)$ be a polynomial with $f^\prime(0) \neq 0$. Then we can choose $c \in \C$ and $a_1, a_2, \ldots \in \C$ such that 
\begin{eqnarray*}
\sum a_i (f(z)-c)^i = 0.
\end{eqnarray*}
\end{theorem}

\begin{proof}

The main idea of the proof is to choose the
constant $c$ large such that the constant terms of $(f-c)^i$ grow rapidly. The coefficients of higher degree terms are given by sums of terms that grow even faster. So
if the norm of the higher degree coefficients are at least as large as the norm of the individual terms that are summed then we can use Lemma \ref{orthogonal} to complete the proof. It turns out that by picking
the constant $c$ carefully we can make sure that the differences between the arguments of all the coefficients of $f^i$ become small and the norms of the coefficients will grow sufficiently fast.

A polynomial $f(z)$ is a root of the powerseries
$\sum a_i (T-c)^i$ if and only if $\g f(\g^{-1} z)$ is a root of
the powerseries $\sum a_i (T - \g c)^i$. Therefore we can choose
$\g$ such that the coefficient of the highest degree term of $\g
f(\g^{-1} z)$ is $1$. From here on we assume that $f(z)$ is monic.

Let us introduce some notation. We will write $(f(z)-c)^i =
\sum_{j=1}^{l(i)} b_j^i z^j$. In particularly we have $f(z)-c =
b_0^1 + b_1^1 + \ldots + z^d$. Note that choosing $c$ is equivalent to choosing $b_0^1$ which is what we will refer to from now on.

We first show that we can choose arbitrarily large $b_0^1$  such that
$\lim_{i \rightarrow \infty} \Arg(\frac{b_0^i}{b_1^i}) = 0$.
First of all, let us look at the rate of growth of the constant
coefficients $b_0^i$. When we choose $b$ large enough then the
sequence $b_0^1, b_0^2, b_0^3, \ldots$ escapes to infinity. We
have that $b_0^{i+1} = (b_0^i)^d + b_{d-1}^{1} (b_0^i)^{d-1} +
l.o.t.$. Similarly we have that $b_1^{i+1} = d (b_0^i)^{d-1}b_1^i +
(d-1) (b_0^i)^{d-2}b_1^i +  l.o.t.$.

We see that there are uniformly bounded error terms $E_i$ and $\tilde{E}_i$ such that

\begin{align}\label{first}
b_0^{i+1} = (b_0^i)^{d-2} (b_0^i + E_i) b_0^i,
\end{align}
and

\begin{align}\label{second}
b_1^{i+1} = d (b_0^1)^{d-2} (b_0^i + \tilde{E}_i) b_1^i.
\end{align}

Therefore we get

\[ \Arg(\frac{b_1^{i+1}}{b_0^{i+1}}) = \Arg \left(\frac{b_0^i + \tilde{E}_i}{b_0^i + E_i}\right) + \Arg(\frac{b_1^i}{b_0^i}).\\
= \Arg(1 + \frac{\hat{E}_i}{b_0^i}) + \Arg(\frac{b_1^i}{b_0^i}).    \]

As the terms $b_0^i$ grow exponentially we see that

\[ \lim_{i \rightarrow \infty}\Arg(\frac{b_1^{i+1}}{b_0^{i+1}}) \]

exists and is very close to $\Arg(\frac{b_1^i}{b_0^i})$ when $b_0^1$ is chosen large. Moreover, the limit depends continuously on the choice of $b_0^1$, so by varying the argument of $b_0^1$ we can make sure that $\lim\Arg(\frac{b_1^{i}}{b_0^{i}}) = 0$.

We now fix $b_0^1$ large such that $b_0^i \rightarrow \infty$ and $\lim\Arg(\frac{b_1^{i+1}}{b_0^{i+1}}) = 0$.
It follows from Equations \ref{first} and \ref{second} that for large $i$ we have

\[ \left|\frac{b_1^{i+1}}{b_0^{i+1}}\right| \approx d \left|\frac{b_1^{i}}{b_0^{i}}\right|. \]

Hence we get that for $i$ large $|b_1^i| \gg |b_0^i|$.
Similarly, we have that

\[b_2^{i+1} = d(b_0^1)^{d-1}b_2^i + {d\choose 2}(b_0^i)^{d-2}(b_1^i)^2 + l.o.t. .\]

Here the second term on the right hand side gives a much faster growth than the first term. Comparing it with the rate of growth of $b_1^{i}$, we see that $|b_2^i| \gg |b_1^i|$. Moreover, as the term ${d\choose 2}(b_0^i)^{d-2}(b_1^i)^2$ determines the growth of the terms $b_2^{i+1}$ and the arguments of $b_0^i$ and $b_1^i$ are very close for large $i$, we see that

\[ \lim_{i \rightarrow \infty} \Arg (\frac{b_2^{i}}{b_0^{i}}) = 0. \]

The argument is identical for $b_j^i$ and $j$ larger than $2$. The coefficient $b_j^{i+1}$ is given by a very large sum. However, all the relevant terms have arguments that are very close (by induction), and it follows that

\[ \lim_{i \rightarrow \infty} \Arg (\frac{b_j^{i}}{b_0^{i}}) = 0. \]

Moreover we see that $|b_j^i| \gg |b_{j-1}^i|$. By Lemma \ref{orthogonal} it follows from that there is a sequence $a_1, a_2, \ldots$
such that $\sum a_i b_j^i = 0$ for every $j$. This means exactly that the
polynomial $f$ is a root of the powerseries $\sum a_i (T-c)^i$.

\end{proof}

\end{document}